\documentclass{amsart}

\usepackage{amsmath, amssymb, mathrsfs, amsthm, xifthen, verbatim, color, mathrsfs}

\usepackage[arrow, matrix]{xy}

\definecolor{DarkBlue}{rgb}{0,0.2,0.6}
\definecolor{PinkPurple}{rgb}{0.8,0.3,0.3}

\usepackage[pdftex,
            pdfauthor={Mehdi Ghasemi},
            pdftitle={Integral representation of a linear functional on function spaces},
            pdfsubject={Integral representation},
            pdfkeywords={Integration, Measure, Topology},
            pdfproducer={Latex with hyperref},
            pdfcreator={PDFLaTeX},
            linkcolor=DarkBlue,
            citecolor=PinkPurple,
            colorlinks=true]{hyperref}

\newtheorem{thm}{Theorem}[section]
\newtheorem{lemma}[thm]{Lemma}
\newtheorem{prop}[thm]{Proposition}
\newtheorem{crl}[thm]{Corollary}

\theoremstyle{definition}
\newtheorem{dfn}[thm]{Definition}
\newtheorem{exm}[thm]{Example}

\newtheorem{rem}[thm]{Remark}

\newcommand{\reals}{\mathbb{R}}
\newcommand{\naturals}{\mathbb{N}}

\newcommand{\rx}{\sgr{\mathbb{R}}{\ux}}

\newcommand{\ux}{\underline{X}}

\newcommand{\supp}{\text{supp}}
\newcommand{\spn}{\text{Span}}

\newcommand{\dm}[1]{\preccurlyeq_{#1}}

\newcommand{\sgr}[2]{#1[#2]}
\newcommand{\ringsop}[2]{\sum #1^{#2}}

\newcommand{\norm}[2]{\|\ifthenelse{\isempty{#2}}{\cdot}{#2}\|_{#1}}
\newcommand{\cl}[2]{\overline{#2}^{\ifthenelse{\isempty{#1}}{}{#1}}}
\newcommand{\Cnt}[2]{\mathrm{C}_{\ifthenelse{\isempty{#1}}{}{#1}}(#2)}
\newcommand{\Psd}[2]{\mbox{Psd}_{\ifthenelse{\isempty{#1}}{}{#1}}(#2)}
\newcommand{\Bnd}[2]{\mbox{B}_{\ifthenelse{\isempty{#1}}{}{#1}}(#2)}
\newcommand{\Sp}[2]{\mathfrak{sp}_{\ifthenelse{\isempty{#1}}{}{#1}}(#2)}
\newcommand{\map}[3]{#1:#2\longrightarrow #3}

\begin{document}

\title[Integral representation of functionals]{Integral representation of linear functionals on function spaces}
%
\author[M. Ghasemi]{Mehdi Ghasemi}
%
%
%
\address{School of Physical \& Mathematical Sciences,\newline\indent
Nanyang Technological University,\newline\indent
21 Nanyang Link, 637371, Singapore}
\email{mghasemi@ntu.edu.sg}
\date{\today}
\begin{abstract}
Let $A$ be a vector space of real valued functions on a non-empty set $X$ and $\map{L}{A}{\reals}$ a linear functional.
Given a $\sigma$-algebra $\mathcal{A}$, of subsets of $X$, we present a necessary condition for $L$ to be representable 
as an integral with respect to a measure $\mu$ on $X$ such that elements of $\mathcal{A}$ are $\mu$-measurable.
This general result then is applied to the case where $X$ carries a topological structure and $A$ is a family of continuous
functions and naturally $\mathcal{A}$ is the Borel structure of $X$. 
As an application, short solutions for the full and truncated $K$-moment problem are presented. An analogue of 
Riesz--Markov--Kakutani representation theorem is given where $\Cnt{c}{X}$ is replaced with whole $\Cnt{}{X}$. 
Then we consider the case where $A$ only consists of bounded functions and hence is equipped with $\sup$-norm.
\end{abstract}
\maketitle
\section{Introduction}
A positive linear functional on a function space $A\subseteq\reals^X$ is a linear map $\map{L}{A}{\reals}$ which assigns a 
non-negative real number to every function $f\in A$  that is globally non-negative over $X$.
The celebrated Riesz--Markov--Kakutani representation theorem states that every positive functional on the space of continuous
compactly supported functions over a locally compact Hausdorff space $X$, admits an integral representation with
respect to a regular Borel measure on $X$. In symbols $L(f)=\int_Xf~d\mu$, for all $f\in\Cnt{c}{X}$. Riesz's original result 
\cite{Riesz} was proved in 1909, for the unit interval $[0,1]$. Then in 1938, Markov extended Riesz's result to some non-compact
spaces \cite{Markov}. Later in 1941, Kakutani \cite{Kakutani} proved the result for compact Hausdorff spaces. 
Some authors such as Pedersen \cite[\S 6.1]{Ped}, take the reverse approach. Instead of defining an integral operator based on a
pre-existing measure, he starts with a special type linear functionals. He calls the class of positive linear functionals 
on $\Cnt{c}{X}$, the ``Radon integrals'' and then fixing such a functional, defines integrable functions and the measure which 
corresponds to $L$. This approach has introduced and studied by P. J. Daniell in 1918 \cite{Daniell} and nowadays is known
as Daniell integral theory.

Perhaps one of the most well-known consequences of Riesz'z result is its application to solve the classical moment problem, given
by Haviland in 1936. Haviland published series of two papers \cite{Hav01, Hav02}, and gave a complete solution for the $K$-moment 
problem. Let $\map{L}{\rx=\reals[X_1,\dots,X_n]}{\reals}$ be a linear functional on the space of real polynomials in $n$ variables and 
$K\subseteq\reals^n$ a closed set. The classical $K$-moment problem asks when $L$ admits an integral representation with respect 
to a Borel measure supported on $K$. Haviland proved that such a measure exists if and only if $L$ maps every non-negative polynomial
on $K$ to a non-negative real number. Let us demystify the connection and similarity between these to results.

Fix a subspace $A\subseteq\reals^X$ where $X$ is a locally compact Hausdorff space and $K\subseteq X$, a closed set. The set of all
functions $f\in A$ that are non-negative on $K$ will be denoted by $\Psd{A}{K}$:
\[
	\Psd{A}{K}:=\{f\in A~:~f(x)\ge0\quad\forall x\in K\}.
\]
Let $\map{L}{A}{\reals}$ be a linear functional. For $A=\Cnt{c}{X}$, Riesz--Markov--Kakutani's result is the following statement:
\begin{thm}\label{RMK}
There exists a Borel measure on $X$ representing $L$ if and only if $L(\Psd{\Cnt{c}{X}}{X})\subseteq[0,\infty)$.
\end{thm}
Also, for $K\subseteq\reals^n=X$ and $A=\rx$, Haviland's theorem is the following:
\begin{thm}\label{HAV}
There exists a Borel measure on $K$ representing $L$ if and only if $L(\Psd{\rx}{K})\subseteq[0,\infty)$.
\end{thm}
The proof of Haviland's theorem relies on Riesz--Markov--Kakutani's result and involves the following steps:
\begin{enumerate}
	\item{\label{st1}
	Extend $L$ positively to the algebra $\check{A}$ consisting of all continuous functions, bounded by a polynomial on $K$.
	}
	\item{\label{st2}
	Showing that $\check{A}$ contains $\Cnt{c}{K}$.
	}
	\item{\label{st3}
	Applying Riesz--Markov--Kakutani representation theorem to the restriction of the map from step \eqref{st1} on $\Cnt{c}{K}$.
	}
	\item{\label{st4}
	Proving that the measure obtained in step \eqref{st3} also represents $L$ on $\rx$.
	}
\end{enumerate}
Suppose that $\mu$ is a finite Borel measure on $X$ and let $\norm{\mu,1}{}$ be the seminorm induces by $\mu$ on $L_1(\mu)$.
One can show that the characteristic function of every Borel set, can be approximated by compactly supported functions in
$\norm{\mu,1}{}$. Therefore $\Cnt{c}{X}$ is dense in $\Cnt{c}{X}\oplus S_{\mathcal{B}(X)}$ where $S_{\mathcal{B}(X)}$ 
is the algebra of the \textit{simple functions}, generated by characteristic functions of all Borel subset of $X$. 
On the other hand, simple functions are dense in $L_1(\mu)$ which implies the denseness of $\Cnt{c}{X}$ in $L_1(\mu)$.

In section \ref{IRPSD}, we generalize this idea to any abstract measure structures on $X$. In lemma \ref{msr-rep}\eqref{msr-rep-4}
we prove an analogue of theorem \ref{RMK} and will use this result to prove theorem \ref{gen-chq} which is a general version of 
Haviland's theorem \ref{HAV}. 

In section \ref{App2Cnt}, we consider perhaps the most interesting case, where the underlying domain
consists of continuous functions. We presents short proofs for results of Choquet (Corollary \ref{hvlnd1}) and Marshall (Corollary 
\ref{Mar-Hav}). Also we apply the results of section \ref{IRPSD} to give a simple solution for truncated moment problem in \S\ref{TMP}.
The Riesz--Markov--Kakutani representation theorem is stated for $\Cnt{c}{X}$ which easily generalizes to $\Cnt{0}{X}$, the space of 
functions vanishing at infinity. In \S\ref{Riesz-Cnt} we investigate the case where the domain is extended to whole $\Cnt{}{X}$ and
show that in this case, the support of representing measure has to be compact set (Theorem \ref{gen-sig-Riesz}).

In section \ref{lmc2d}, we focus on subalgebras of $\ell^{\infty}(X)$, the algebra of globally bounded functions. This algebra 
naturally is equipped with a norm topology, where the norm is defined by $\norm{X}{f}=\sup_{x\in X}|f(x)|$.
We prove that every continuous positive functional on a subalgebra $A$ of $\ell^{\infty}(X)$, admits an integral representation 
with respect to a unique Radon measure on the Gelfand spectrum $\Sp{\norm{X}{}}{A}$ of $(A,\norm{X}{})$. We also consider the 
possibility of existence of a representing measure on $X$, where $X$ can be realised as a dense subspace of $\Sp{\norm{X}{}}{A}$.

\section{Integral Representation of a Positive Functional}\label{IRPSD}
In this section we study integral representation of a linear functional on a subalgebra of $\reals^X$.
We also prove a variation of Riesz representation theorem which holds for $\sigma$-compact, locally compact spaces and 
replaces $\Cnt{0}{X}$ with $\Cnt{}{X}$ (Theorem \ref{gen-sig-Riesz}).

The following result
plays a key role in this section.
\begin{thm}\label{chq}
Let $W$ be a subspace of an $\reals$-vector space $V$ and $C\subseteq V$ a convex cone. Let $\map{L}{W}{\reals}$ be a functional 
with $L(W\cap C)\ge0$. Then $L$ admits an extension $\tilde{L}$ on $W_C$ such that $\tilde{L}(W_C\cap C)\ge0$. Here
\[
	W_C=\{v\in V~:~\pm v\in C+W\}.
\]
\end{thm}
\begin{proof}
It is clear that $W_C$ is a subspace of $V$ containing $W$. We show that the function $p(v)=-\sup\{L(w):w\in W\wedge v-w\in C\}$ 
which is defined on $W_C$ is a sublinear function such that $p|_W=-L$. 
To see this note that there are $w,w'\in W$ and $c,c'\in C$ such that $v=w+c=w'-c'$. Thus $w'-w=c+c'\in C\cap W$ and hence 
$L(w'-w)\ge0$ or $L(w)\leq L(w')$. Therefore the set $\{L(w):w\in W\wedge v-w\in C\}$ is non-empty and bounded above. Hence $p(v)$
exists. Clearly $p(\lambda v)=\lambda p(v)$, so it remains to show that $p(v+v')\leq p(v)+p(v')$. If $v-w\in C$ and $v'-w'\in C$,
then $(v+v')-(w+w')\in C+C=C$. Thus $-p(v)-p(v')\leq-p(v+v')$ or equivalently $p(v+v')\leq p(v)+p(v')$. For every $v\in W$, 
$0=v-v\in C$, therefore $p(v)\leq-L(v)$. Also for every $w\in W$ with $v-w\in C$, we have $L(w)\leq L(v)$, because $L(W\cap C)\ge0$.
Therefore $-L(v)\leq p(v)$ which proves $p|_W=-L$.

Applying Hahn-Banach theorem, $-L$ admits an extension $-\tilde{L}$ to $W_C$ such that
$-\tilde{L}(v)\leq p(v)$ on $W_C$. For $c\in C\cap W_C$, $p(c)\leq0$ and hence $0\leq-p(c)\leq\tilde{L}(c)$ as desired (For the
original result, see \cite[Theorem 34.2]{chq}).
\end{proof}
\begin{dfn}
For a vector subspace $A$ of $\reals^X$, 
\begin{enumerate}
\item{
The set of all elements of $A$ that are positive on $X$ is denoted by $\Psd{A}{X}$. 
\[
	\Psd{A}{X}=\{a\in A~:~\forall x\in X\quad a(x)\ge0\};
\]
}
\item{
The \textit{lattice hull} of $A$ denote by $\check{A}$ is defined to be 
\[
	\check{A}:=\{f\in\reals^X~:~\exists g\in A\quad|f|\leq|g|\}.
\]
$A$ is said to be \textit{lattice complete} if $\check{A}=A$.
}
\item{
Let $B\subseteq\reals^X$ be a vector subspace of $\reals^X$. For $f,g\in A$, we say \textit{$g$ is dominated by $f$ with respect to $B$} 
and write $g\dm{B}f$ if
\[
	\forall\epsilon>0~\exists h\in B\quad |g|\leq\epsilon|f|+h.
\]
We say $A$ is \textit{$B$-adapted} if $\forall g\in A~\exists f\in A$ such that $g\dm{B}f$.
}
\item{
For an algebra $\mathcal{A}$ of subsets of $X$, we denote by $S_{\mathcal{A}}$, the algebra generated by characteristic functions of 
elements of $\mathcal{A}$\footnote{$\chi_Y$ denotes the character function of the set $Y$ defined by $\chi_Y(x)=1$ if $x\in Y$ and $0$ 
otherwise.}.
}
\item{
The set of all (infinite) sums of simple functions that are bounded above by an element of $A$ in absolute value will be denoted by
$\cl{A}{S_{\mathcal{A}}}$,
\[
	\cl{A}{S_{\mathcal{A}}}:=\{\sum_{i=1}^{\infty}\lambda_i\chi_{X_i}~:~X_i\in\mathcal{A}\wedge\exists a\in A\quad|\sum_{i=1}^{\infty}\lambda_i\chi_{X_i}|\leq a\}.
\]
It is clear that $\cl{A}{S_{\mathcal{A}}}$ is a vector space.
}
\end{enumerate}
\end{dfn}
\begin{lemma}\label{msr-rep}
Let $\mathcal{A}$ be a $\sigma$-algebra of subsets of $X$, $B$ a lattice complete algebra of $\mathcal{A}$-measurable functions for 
which every $\chi_Y$, $Y\in\mathcal{A}$ is bounded above in $B$ and $\map{L}{B}{\reals}$ a positive functional. Then 
\begin{enumerate}
	\item{\label{msr-rep-1}
	$L$ extends positively to a linear functional $\map{\bar{L}}{B+S_{\mathcal{A}}}{\reals}$ (and also to $B+\cl{B}{S_{\mathcal{A}}}$);
	}
	\item{\label{msr-rep-2}
	the function $\rho_L(f)=L(|f|)$ (resp. $\rho_{\bar{L}}=\bar{L}(|f|)$) defines a seminorm on $B$ (resp. on $B+S_{\mathcal{A}}$);
	}
	\item{\label{msr-rep-3}
	every function $f\in B$ can be approximated by elements of $\cl{B}{S_{\mathcal{A}}}$ from below, in $\rho_{\bar{L}}$;
	}
	\item{\label{msr-rep-4}
	if $B$ is dense in $(B+S_{\mathcal{A}},\rho_{\bar{L}})$, then there exists a measure $\mu$ on $(X,\mathcal{A})$ such that
	\[
		\forall f\in\Psd{B}{X}\quad L(f)=\int f~d\mu.
	\]
	}
\end{enumerate}
\end{lemma}
\begin{proof}
\eqref{msr-rep-1}
Let $V=B+S_{\mathcal{A}}$ and $C=\Psd{V}{X}$. It suffices to show that $V=B+C$, then the conclusion follows from theorem \ref{chq}.
A typical element of $V$ is of the form $b+\sum_{i=1}^n\lambda_i\chi_{Y_i}$, $b\in B$, $Y_i\in\mathcal{A}$ and $\lambda_i\in\reals$, 
for $i=1,\dots,n$. By assumption, for each $i=1,\dots,n$, there is $b_i\in B$ with $\chi_{Y_i}\leq b_i$. 
Take $b'=\sum_{i=1}^n|\lambda_i|b_i$, then $b'\pm\sum_{i=1}^n\lambda_i\chi_{Y_i}\in C$ and
\[
	b\pm\sum_{i=1}^n\lambda_i\chi_{Y_i}=(b-b')+(b'\pm\sum_{i=1}^n\lambda_i\chi_{Y_i})\in B+C.
\]
Therefore $V=B+C$ as desired. Now let $V=B+\cl{B}{\mathcal{A}}$ and $C=\Psd{V}{X}$. For $b\pm\sum_{i=1}^{\infty}\lambda_i\chi_{X_i}$,
there exists $b'\in B$ such that $|\sum_{i=1}^{\infty}\lambda_i\chi_{X_i}|<b'$, so $b'\pm\sum_{i=1}^{\infty}\lambda_i\chi_{X_i}\in C$
and
\[
	b+\sum_{i=1}^{\infty}\lambda_i\chi_{X_i}=(b-b')+(b'\pm\sum_{i=1}^{\infty}\lambda_i\chi_{X_i})\in B+C,
\]
and the conclusion follows from theorem \ref{chq}.

\eqref{msr-rep-2}
This is clear. Since $B$ is lattice complete, it contains absolute value of its elements 
(note that absolute value of an $\mathcal{A}$-measurable function is $\mathcal{A}$-measurable as well). Therefore by positivity of 
$L$ we see $\rho_L(rf)=|r|L(|f|)=|r|\rho_L(f)$ and $\rho_L(f+g)=L(|f+g|)\leq L(|f|+|g|)=L(|f|)+L(|g|)=\rho_L(f)+\rho_L(g)$.

\eqref{msr-rep-3}
Fix $f\in B$ and $\epsilon>0$. First suppose that $f(X)$ is bounded. This means that there are $a<b\in\reals$ such that for every
$x\in X$, $a\leq f(x)<b$. Let $n\ge1$ be an integer and define
\[
	X_k:=f^{-1}\left([a+(k-1)\frac{b-a}{n},a+k\frac{b-a}{n})\right),\quad k=1,\dots,n.
\]
Each $X_k$ is $\mathcal{A}$-measurable. Take $\phi=\frac{b-a}{n}\sum_{k=1}^n(k-1)\chi_{X_k}$, this is clear that $\phi(x)\leq f(x)$ 
for all $x\in X$ and $0\leq f(x)-\phi(x)<\frac{b-a}{n}$. Therefore $0\leq\bar{L}(f-\phi)\leq\frac{b-a}{n}\bar{L}(1)$. Choose $n$
big enough such that $\frac{b-a}{n}\bar{L}(1)<\epsilon$, we will have $\rho_{\bar{L}}(f-\phi)=\bar{L}(f-\phi)<\epsilon$.

Now take an arbitrary $f\in B$ and let $\{[a_i,b_i)\}_{i\in\naturals}$ be a partition of $\reals$ and $\epsilon>0$. Since $B$ is
lattice complete and $|\chi_{[a_i,b_i)}f|\leq|f|$, we see that $\chi_{[a_i,b_i)}f\in B$.
For every $i\in\naturals$, there exists $\phi_i\in S_{\mathcal{A}}$ such that $\rho_{\bar{L}}(\chi_{[a_i,b_i)}f-\phi_i)<2^{-i}\epsilon$
and $\phi_i\leq\chi_{[a_i,b_i)}f$. Let $\psi=\sum_{i=1}^{\infty}\phi_i$, then $|\psi|\leq|f|$ on $X$, so $\psi\in\cl{B}{S_{\mathcal{A}}}$
and 
\[
\begin{array}{lcl}
	\rho_{\bar{L}}(f-\psi) & = & \rho_{\bar{L}}(\sum_{i=1}^{\infty}\chi_{[a_i,b_i)}f-\phi_i)\\
		& \leq & \sum_{i=1}^{\infty}\rho_{\bar{L}}(\chi_{[a_i,b_i)}f-\phi_i)\\
		& < & \sum_{i=1}^{\infty}2^{-i}\epsilon\\
		& = & \epsilon,
\end{array}
\]
as desired.

\eqref{msr-rep-4}
Let $\bar{L}$ be the functional from part \eqref{msr-rep-1} and $\{Y_n\}_n\subset\mathcal{A}$ a sequence of pairwise disjoint sets.
Then $Y=\cup_{n}Y_n\in\mathcal{A}$ and $\chi_Y=\sum_{n=1}^{\infty}\chi_{Y_n}$. Since $\bar{L}$ is positive, 
$0\leq\sum_1^n\bar{L}(\chi_{Y_i})+\bar{L}(\sum_{n+1}^{\infty}\chi_{Y_i})=\bar{L}(\chi_Y)$. 
The sequence $(\bar{L}(\sum_{n+1}^{\infty}\chi_{Y_i}))_n$ is positive and decreasing and hence convergent. Thus for any $\epsilon>0$, 
there exists $N>0$ such that for any $n>m>N$,
\[
\begin{array}{lcl}
	\bar{L}(\sum_{m+1}^{\infty}\chi_{Y_i})-\bar{L}(\sum_{n+1}^{\infty}\chi_{Y_i}) & = & \bar{L}(\sum_{m+1}^{n}\chi_{Y_i}) \\
	 & = & \sum_{m+1}^n\bar{L}(\chi_{Y_i})\\
	 & < & \frac{\epsilon}{2},
\end{array}
\]
hence $\sum_{m+1}^{\infty}\bar{L}(\chi_{Y_i})<\epsilon$. Therefore $\lim_{n\rightarrow\infty}\sum_{n+1}^{\infty}\bar{L}(\chi_{Y_i})=0$.
So the set function defined on $\mathcal{A}$ by $\mu(A)=\bar{L}(\chi_Y)$ is indeed a measure on $\mathcal{A}$.

Now, for any $\mu$-measurable function $f\ge0$,
\[
\begin{array}{lcl}
	\int f~d\mu & = & \sup\{\int\phi~d\mu:\phi\in S_{\mathcal{A}}\textrm{ and }\phi\leq f\}\\
	 & = & \sup\{\bar{L}(\phi):\phi\in S_{\mathcal{A}}\textrm{ and }\phi\leq f\}.
\end{array}
\]
By density of $B$ in $(B+S_{\mathcal{A}},\rho_{\bar{L}})$ and applying part \eqref{msr-rep-3}, for any $f\in B$ we have:
\[
	\sup\{\bar{L}(\phi):\phi\in S_{\mathcal{A}}\textrm{ and }\phi\leq f\} = 
	\sup\{L(g)\in B:g\in B\textrm{ and }g\leq f\}=L(f),
\]
therefore $L(f)=\int f~d\mu$.
\end{proof}
\begin{prop}\label{ext-alg}
If $A$ is an algebra then every positive functional $\map{L}{A}{\reals}$ admits a positive extension to $\check{A}$.
\end{prop}
\begin{proof}
Note that for $g\in A$, $|g|\leq\frac{g^2+1}{2}$. Therefore, for every $f\in\check{A}$ there exists $g\in\Psd{A}{X}$ such that 
$|f|\leq g$ and hence $f=-g+(g+f)\in A+\Psd{\check{A}}{X}$. Applying theorem \ref{chq} the conclusion follows.
\end{proof}
\begin{thm}\label{gen-chq}
Let $\mathcal{A}$ a $\sigma$-algebra of subsets of $X$, $A\subseteq\reals^X$ an algebra of $\mathcal{A}$-measurable functions, 
$\map{L}{A}{\reals}$, a positive functional and $B\subseteq\check{A}$ a lattice complete subspace, such that
\begin{enumerate}
	\item{
	$A$ is $B$-adapted and,
	}
	\item{
	$B$ is $\rho_{\tilde{L}}$ dense in $B+S_{\mathcal{A}}$.
	}
\end{enumerate}
Then there is a measure $\mu$ on $X$ such that $(X,\mathcal{A},\mu)$ is a measure space and
\[
	\forall f\in\check{A}\quad \tilde{L}(f)=\int f~d\mu.
\]
\end{thm}
\begin{proof}
Since $L$ is positive by theorem \ref{chq} and proposition \ref{ext-alg}, it extends positively to 
$\map{\tilde{L}}{\check{A}}{\reals}$. Abusing the notations, we also denote the extension of $\tilde{L}|_B$ to 
$B+S_{\mathcal{A}}$ by $\tilde{L}$. Since $B$ is dense in $(B+S_{\mathcal{A}},\rho_{\tilde{L}})$, by lemma 
\ref{msr-rep}, there exists a measure $\mu$ such that $(X,\mathcal{A},\mu)$ is a measure space and 
$\tilde{L}(f)=\int f~d\mu$ for all $f\in B+S_{\mathcal{A}}$. For every $f\in\check{A}$,
\[
\begin{array}{lcl}
	\int f~d\mu & = & \sup\{\int\phi~d\mu~:~\phi\in\mathcal{A}\textrm{ and }\phi\leq f\}\\
	 & = & \sup\{\int b~d\mu~:~b\in B\textrm{ and }b\leq f\}\\
	 & \leq & \tilde{L}(f).
\end{array}
\]
Let $T(f)=\tilde{L}(f)-\int f~d\mu$, then $T|_B=0$. Since $A$ is $B$-adapted, for every $g\in\Psd{\check{A}}{X}$ there exists 
$f\in\Psd{\check{A}}{X}$ such that 
\[
	\forall\epsilon>0~\exists h\in B\quad g\leq\epsilon f+h.
\]
Therefore $0\leq T(g)\leq\epsilon T(f)+T(h)=\epsilon T(f)$. Letting $\epsilon\rightarrow0$, we get $T(g)=0$ for all $g\in\check{A}$
and hence for all $g\in\check{A}$, $\tilde{L}(g)=\int g~d\mu$.
\end{proof}
\begin{rem}\label{sbspc-int-rep}
In theorem \ref{gen-chq}, suppose that $A\subseteq V\subseteq\reals^X$ are just vector spaces and $\map{L}{V}{\reals}$ 
a positive linear functional and there is a lattice complete subspace $B\subseteq\check{A}$ which is dense in 
$(B+S_{\mathcal{A}},\rho_L)$ and for every $g\in A$ there exists $f\in V$ such that for all $\epsilon>0$, $|g|\leq\epsilon|f|+h$ 
for some $h\in B$, then the same argument proves that $L$ admits an integral representation on $A$.
\end{rem}
\section{Algebra of continuous functions}\label{App2Cnt}
Perhaps the most interesting case happens when the function algebra is considered to be a subalgebra of continuous functions
with respect to a given topology on $X$. A natural question in this case is when the resulting measure is Radon or Borel?
The main motivation to study representation of a positive functional on a subalgebra of continuous functions to us is the
classical moment problem, where the subalgebra is taken to be polynomials on a subset of $\reals^n$. In this section we 
restate a slightly more general version of Haviland's solution for multidimensional moment problem.
\begin{crl}[Choquet \cite{chq}]\label{hvlnd1}
Let $X$ be a locally compact Hausdorff space, $A\subseteq\Cnt{}{X}$ a $\Cnt{c}{X}$-adapted space which separates points of $X$
and $\map{L}{A}{\reals}$ a positive linear functional. Then $L$ is representable via a positive Borel measure on $X$.
\end{crl}
\begin{proof}
Take $\mathcal{A}=\mathcal{B}(X)$, the Borel algebra of $X$. Let $g\in\Cnt{c}{X}$, we show that $g\in\check{A}$. Since $A$ separates
points of $X$, for every $x\in X$, there exists $g_x\in\Psd{A}{X}$ such that $|g(x)|<g_x(x)$ holds on an open neighbourhood $V_x$ of $x$.
The family $(V_x)_{x\in X}$ forms an open cover for $\supp(g)$ which is compact, hence there exist $x_1,\dots,x_n$, such that
$\supp(g)\subseteq\cup_i^nV_{x_i}$.
Therefore $|g|<\sum_1^ng_{x_i}$ and $g\in\check{A}$. Thus $\tilde{L}$ is positive on $\Cnt{c}{X}$ and by Riesz--Markov--Kakutani 
representation theorem, there exists a positive Borel measure $\mu$ on $X$ such that $\tilde{L}(f)=\int f~d\mu$ on $\Cnt{c}{X}$ and 
since $A$ is $\Cnt{c}{X}$-adapted, theorem \ref{gen-chq} implies that $\tilde{L}(f)=\int f~d\mu$ for all $f\in\check{A}$.
\end{proof}
\begin{rem}
In Corollary \ref{hvlnd1}, one could easily show that for every compact subset $K$ of $X$, $\chi_K$ is a uniform limit of a 
sequence in $\Cnt{c}{X}$ and hence $\Cnt{c}{X}$ is dense in $\Cnt{c}{X}+S_{\mathcal{B}(X)}$. Therefore the measure obtained 
from lemma \ref{msr-rep}(\ref{msr-rep-4}) has to be a Borel measure.
\end{rem}
\begin{crl}[Marshall \cite{Mar02}]\label{Mar-Hav}
Suppose $A$ is an $\reals$-algebra, $X$ is a  Hausdorff space, and $\map{\hat{~}}{A}{\Cnt{}{X}}$ is an $\reals$-algebra homomorphism such 
that for some $p\in A$, $\hat{p}\ge0$ on $X$, the set $X_i=\hat{p}^{-1}([0,i])$ is compact for each $i=1,2,\cdots$. Then for every linear 
functional $\map{L}{A}{\reals}$ satisfying $L(\Psd{\hat{A}}{X})\subseteq[0,\infty)$, there exists a Radon measure $\mu$ on $X$ such that 
\[
	\forall a\in A \quad L(a)=\int_X\hat{a}~d\mu.
\]
\end{crl}
\begin{proof}
It suffices to show that $\hat{A}$ is a $\Cnt{c}{X}$-adapted space.
Take $g\in A$ with $\hat{g}\ge0$ on $X$ and let $f=g+p$. The set $X_n=\hat{f}^{-1}([0,n])\subseteq\hat{p}^{-1}([0,n])$ is compact and 
$X_n\subseteq X_{n+1}$. Also for each $n$, the set $Y_n=X_{n+1}\cap\hat{f}^{-1}([n+\frac{1}{2},\infty))$ is closed and disjoint from 
$X_n$.

By Urysohn's lemma there exists a continuous function $\map{e_n}{X_{n+1}}{[0,1]}$ such that $e_n|_{Y_n}=0$ and $e_n|_{X_n}=1$.
Extending $e_n$ to $X$ by defining $e_n(x)=0$ off $X_{n+1}$, we have $g_n=\hat{g}e_n\in\Cnt{c}{X}$.

Fix $\epsilon>0$ and $N>0$ with $\frac{1}{N}<\epsilon$, we have
\[
	\hat{g}-g_N\leq\epsilon \hat{f}^2
\]
on $X$ or $\hat{g}\leq\epsilon\hat{f}^2+g_N$ and hence $\hat{g}\dm{\Cnt{c}{X}}\hat{f}^2$ as desired. Now, theorem \ref{gen-chq} applies 
and the existence of such a measure follows.
\end{proof}
\subsection{Application to Classical Moment Problem}
One can apply Corollary \ref{Mar-Hav} to give a short solution for the classical $K$-moment problem:
\begin{crl}[Haviland \cite{Hav01, Hav02}]
Let $K$ be a closed subset of $\reals^n$ and $\map{L}{\rx}{\reals}$. Then $L$ admits and integral representation with respect to a 
positive Radon measure supported on $K$ if and only if $L(\Psd{}{K})\subseteq[0,\infty)$.
\end{crl}
\begin{proof}
Take $p=\sum_{i=1}^n X_i^2\in\rx$, clearly $p^{-1}([0,n])$ is compact for each $n$. Thus Corollary \ref{Mar-Hav} applies.
\end{proof}
\subsection{Application to Truncated Moment Problem}\label{TMP}
Let $A$ be a subspace of $\Cnt{}{X}$ and $\mathfrak{B}\subset A$ a basis for $A$. Let $V=A+\spn\{1+p^2:p\in\mathfrak{B}\}$ 
and suppose that for every $g\in A$ there exists $f\in V$ such that for every $\epsilon>0$ there is $h\in\Cnt{c}{X}$ satisfying 
$|g|\leq\epsilon|f|+h$. Then according to Remark \ref{sbspc-int-rep}, any positive functional $\map{L}{V}{\reals}$, admits an integral 
representation with respect to a Borel measure on $X$ over $A$. One can improve this situation in the case where $A$ is a finite 
dimensional subspace of the polynomials $\rx$.
\begin{prop}[Curto--Fialkow \cite{Curto-Fialkow01}]\label{C-F}
Let $\map{L}{\rx_{2d}}{\reals}$, $d\ge1$, be a $K$-positive functional where $K\subseteq\reals^n$. Then there exists a Radon measure 
$\mu$ on $K$ such that for all $p\in\rx_{2d-1}$, $L(p)=\int f~d\mu$.
\end{prop}
Here $\rx_m$ denotes the set of all polynomials of degree at most $m$.
\begin{proof}
Every polynomial $p\in\rx_{2d-1}$, is bounded by a polynomial of degree $2d$ outside of a compact set. In other words, there exists
$h\in\rx_{2d}$ and $h\in\Cnt{c}{K}$ such that $|p|\leq|f|+h$. Therefore Remark \ref{sbspc-int-rep} applies and hence $L|_{\rx_{2d-1}}$
is representable by a Radon measure on $K$.
\end{proof}
As it is proved in \cite[Theorem 2.2]{Curto-Fialkow01}, the combination of proposition \ref{C-F} and Bayer--Teichmann 
\cite[Theorem 2]{Bayer-Teichmann}, implies the following:
\begin{crl}
A functional $\map{L}{\rx_{2d}}{\reals}$ admits a $K$-representing measure if and only if $L$ extends positively to $\rx_{2d+2}$.
\end{crl}
\begin{proof}
See \cite[Theorem 2.2]{Curto-Fialkow01}.
\end{proof}
\subsection{Representation of a Functional on $\mathrm{C}(X)$}\label{Riesz-Cnt}
Now we state a variation of Riesz--Markov--Kakutani representation theorem where $\Cnt{c}{X}$ is replaced by $\Cnt{}{X}$ and
$X$ is assumed to be a $\sigma$-compact space.
We show that a positive linear functional on $\Cnt{}{X}$ is representable via a Radon measure on $X$, but it has to have a 
compact support.
\begin{dfn}
A subalgebra $B\subseteq A$ is called $\sigma$-complete when for any sequence $(b_n)_{n\in\naturals}\subset B$, if $b=\sum_n b_n$ 
exists in $A$ then $b\in B$.
\end{dfn}
\begin{thm}\label{gen-sig-Riesz}
Let $A$ be a $\sigma$-complete subalgebra of $\Cnt{}{X}$ which separates points of $X$. Suppose that $X$ is a $\sigma$-compact, 
locally compact and Hausdorff space and $\map{L}{A}{\reals}$ is a positive linear functional. Then there exists a Radon measure
$\mu$ on $X$ with compact support, such that $L(f)=\int_Xf~d\mu$ for all $f\in A$.
\end{thm}
\begin{proof}
We break the proof into three steps.

\textbf{Step 1.} There is a sequence $(X_n)_{n\in\naturals}$ of compact subsets of $X$ such that $X=\cup_nX_n$ and 
$X_n\subseteq X_{n+1}^{\circ}$.

To see this, choose a countable compact cover $(C_n)_{n\in\naturals}$ for $X$ and set $D_n=\cup_{i=1}^nC_i$. For every $n$, the set $D_n$ 
is compact and since $X$ is locally compact, every $x\in D_n$ has a neighbourhood $V_x$ where $\cl{}{V_x}$ is compact. The collection
$\{V_x~:~x\in D_n\}$ is an open cover for $D_n$ and hence there are $x_1,\dots,x_t\in D_n$ such that $D_n\subseteq\cup_{i=1}^tV_{x_i}$.
Now, let $X_n=\cup_{i=1}^t\cl{}{V_{x_i}}$, the collection $(X_n)_{n\in\naturals}$ is the desired sequence.

\textbf{Step 2.} There exists a compact $K\subseteq X$ such that $f|_K=0$ implies $L(f)=0$.

Assume that such $K$ does not exist. So for every compact $C\subseteq X$, there exists $f_C\in A$ such that $f_C|_C=0$ but $L(f_C)\neq0$.
Substitute $f_C$ with $f_C^2$ if necessary, we can assume that $f_C\ge0$ and hence $L(f_C)>0$. Thus for sufficiently large $N_C$, 
$L(N_Cf_C)>1$. Take a compact cover $(X_n)_{n\in\naturals}$ of $X$ with $X_n\subseteq X_{n+1}^{\circ}$ and denote $N_{X_n}f_{X_n}$ by $g_n$.
For every $x\in X$, $\sum_{n}g_n(x)$ is finite and is continuous according to the choice of $(X_n)_{n\in\naturals}$. 
Therefore $g=\sum_ng_n\in\Cnt{}{X}$ and since $A$ is $\sigma$-complete, $g\in A$. Computing $L(g)$, we get
\[
	\sum_n1<\sum_nL(g_n)\leq L(\sum_ng_n)=L(g)<\infty,
\]
a contradiction. This proves that such a compact $K\subseteq X$ must exists.

\textbf{Step 3.} There exists a linear functional $\tilde{L}$ on $\Cnt{}{K}$ such that 
\[
	\forall f\in A\quad L(f)=\tilde{L}(f|_K).
\]

Let $I(K)=\{f\in\Cnt{}{X}:f|_K=0\}$ and $I_A(K)=I(K)\cap A$. Applying Titze extension theorem, we can identify $\frac{\Cnt{}{X}}{I(K)}$
with $\Cnt{}{K}$. By Step 2, $I_A(K)\subseteq\ker L$, therefore $L$ factors through $\map{\pi}{A}{\frac{A}{I_A(K)}}$.
Moreover the map $\map{\vartheta}{\frac{A}{I_A(K)}}{\frac{\Cnt{}{X}}{I(K)}}$ defined by $f+I_A(K)\mapsto f+I(K)$ is well-defined and injective.
Since $A$ separates points of $X$, $\frac{A}{I_A(K)}$ also separates points of $K$ and by Stone-Weierstrass theorem, image of $\vartheta$ is
dense in $\Cnt{}{K}$. Hence $\bar{L}$ admits an extension $\map{\tilde{L}}{\Cnt{}{K}}{\reals}$ which respects positivity and 
$\tilde{L}(f|_K)=L(f)$ for all $f\in A$ (see Figure \ref{Fig1}).
\begin{figure}
\[
\xymatrix{
	A\ar[r]^{\pi}\ar[dr]_{L} & \frac{A}{I_A(K)}\ar[d]^{\bar{L}}\ar[r]^{\vartheta} & \frac{\Cnt{}{X}}{I(K)}\ar[dl]^{\tilde{L}} \\ 
	& \reals
}
\]
\caption{}\label{Fig1}
\end{figure}
Now Riesz--Markov--Kakutani representation theorem guarantees the existence of a Radon measure $\nu$ on $K$ such that 
\[
	\forall g\in\Cnt{}{K}\quad\tilde{L}(g)=\int g~d\nu.
\]
Extend $\nu$ to a measure $\mu$ on $X$ by defining $\mu(E)=\nu(E\cap K)$, we have
\[
	\forall f\in A\quad L(f)=\int f~d\mu,
\]
as desired.
\end{proof}
\section{Representation of a Functional over a subalgebra of $\ell^{\infty}(X)$}\label{lmc2d}
In this section we consider a functional $\map{L}{A}{\reals}$ where $A$ is a unital commutative $\reals$-algebra for which
there exists an $\reals$-algebras homomorphism $\map{\iota}{A}{\ell^{\infty}(X)}$, and investigate the possibility of 
representing $L$ as an integral with respect to a measure $\mu$ on $X$ such that
\[
	L(a)=\int_X\iota a~d\mu\quad\forall a\in A.
\]
The algebra $\ell^{\infty}(X)$ is naturally equipped with a norm defined as
\[
	\norm{X}{f}=\sup_{x\in X}|f(X)|
\]
for every $f\in\ell^{\infty}(X)$. This induces a seminorm $\norm{}{}$ on $A$ defined by $\norm{}{a}=\norm{X}{\iota a}$.
The relation between locally multiplicatively convex topologies (such as $\norm{X}{}$), positive cones and integral representation
of continuous functionals has been studied in \cite{MGHSKMM}. What follows, is a consequence of \cite[Theorem 3.7]{MGHSKMM}.

Let $d\ge1$ an integer number. We denote by $\ringsop{A}{2d}$, the cone of all finite sums of
$2d^{th}$ powers of elements of $A$; i.e.,
\[
	\ringsop{A}{2d}=\{a_1^{2d}+\cdots+a_m^{2d}~:~a_1,\dots,a_m\in A,~ m\ge1\}.
\]
We denote the set of all $\norm{}{}$-continuous real valued algebra homomorphisms on $A$ by $\Sp{\norm{}{}}{A}$ which is known as
the Gelfand spectrum of the seminormed algebra $(A,\norm{}{})$.
\begin{prop}\label{nrm-sod}
Let $\map{\iota}{A}{\ell^{\infty}(X)}$ be an $\reals$-algebras homomorphism and $d\ge1$ an integer. 
Then $\cl{\norm{X}{}}{\ringsop{A}{2d}}=\Psd{A}{X}$ where $\norm{}{}$ is the pull-back seminorm induced by $\norm{X}{}$ on 
$A$ through $\iota$.
\end{prop}
\begin{proof}
The map $\iota$ induces a function $\map{\iota_*}{X}{\Sp{\norm{}{}}{A}}$ by $\iota_*(x)(a)=\iota a(x)$.\\
\indent\textit{Claim.} $\iota_*X$ is dense in $\Sp{\norm{}{}}{A}$.

Every element $a\in A$ corresponds to a continuous map $\map{\hat{a}}{\Sp{\norm{}{A}}}{\reals}$ defined as $\hat{a}(\alpha)=\alpha(a)$.
If $\iota_*X$ is not dense in $\Sp{\norm{}{}}{A}$, then we can take $\alpha\in\Sp{\norm{}{}}{X}\setminus\cl{}{\iota_*X}$. 
Since $\Sp{\norm{}{}}{A}$ is compact and Hausdorff, there exists $f\in\Cnt{}{\Sp{\norm{}{}}{A}}$ such that $f(\alpha)=1$ and 
$f|_{\iota_*X}=0$. Clearly, $A$ separates points of $\Sp{\norm{}{}}{A}$, by Stone--Weierstrass theorem it is dense in 
$\Cnt{}{\Sp{\norm{}{}}{A}}$. Therefore, for $\epsilon>0$, there exists $a_{\epsilon}\in A$ with $\norm{}{f-a_{\epsilon}}<\epsilon$.
Take $\epsilon>0$ such that $\frac{1-\epsilon}{\epsilon}>1$. Then $|f(\alpha-\alpha(a_{\epsilon})|=|1-\alpha(a_{\epsilon})|<\epsilon$
or $1-\epsilon<|\alpha(a_{\epsilon})|<1+\epsilon$. Also $|f(\iota_*x)-\iota a_{\epsilon}(x)|=|\iota a_{\epsilon}(x)|<\epsilon$ for all
$x\in X$.But
\[
\begin{array}{lcl}
	\sup_{\beta\in\Sp{\norm{}{}}{A}}|\beta(a_{\epsilon})| & \leq & \norm{}{a_{\epsilon}}\\
		& \leq & \sup_{x\in X}|\iota a_{\epsilon}(x)|\\
		& \leq & \epsilon\\
		& < & 1-\epsilon,
\end{array}
\]
and hence $|\alpha(a_{\epsilon})|<1-\epsilon$, a contradiction. This proves the claim.

By \cite[Theorem 3.7]{MGHSKMM}, $\cl{\norm{}{}}{\ringsop{A}{2d}}=\Psd{A}{\Sp{\norm{}{}}{A}}$. By the above claim $\iota_*X$,  
is dense in $\Sp{\norm{}{}}{A}$ and therefore $\hat{a}\ge0$ on $\Sp{\norm{}{}}{A}$ if and only if $\hat{a}\ge0$ on $\iota_*X$ 
or equivalently $\iota a\ge0$ on $X$.
\end{proof}
An application of Banach separation theorem shows that $\cl{\norm{X}{}}{\ringsop{A}{2d}}=\Psd{A}{X}$ holds, 
if and only if every $\norm{}{}$-continuous functional $\map{L}{A}{\reals}$ that is non-negative on $\ringsop{A}{2d}$, 
admits an integral representation with respect to a Radon measure $\mu$ supported on $\Sp{\norm{}{}}{A}$ 
(\cite[Corollary 3.8]{MGHSKMM}), such that:
\[
	L(a)=\int_{\Sp{\norm{}{}}{A}} \hat{a}~d\mu.
\]
Since we are representing $A$ as a subalgebra of $\ell^{\infty}(X)$, it is natural to ask if $L$ comes from a positive
measure $\nu$ on $X$ such that $L(a)=\int_X\iota a~d\nu$.

The answer to the above question is negative in general. We consider the case where $A$ separates points of $X$ and hence 
$\Sp{\norm{}{}}{A}$ contains $X$ as a dense subspace, according to the claim proved in proposition \ref{nrm-sod}.
\begin{exm}
Suppose that $X$ is equipped with a non-compact completely regular topology and $A=\Cnt{b}{X}$. 
Then $\beta X\setminus X\neq\emptyset$. Every homomorphism $\alpha\in\beta X\setminus X$ is a continuous functional which
corresponds to the point-measure $\delta_{\alpha}$ supported at $\alpha$ itself. Clearly $\alpha(f)=\int f~d\delta_{\alpha}$,
and it is not representable with respect to any measure supported on $X$.
\end{exm}
The next example shows that there might exist a non-extremal functional which is representable by a measure on $\Sp{\norm{}{}}{A}$
but not on $X$.
\begin{exm}
Suppose that $Y$ is a separable compact Hausdorff space and $\mu$ be a measure on $Y$ such that each point is of zero measure. 
Let $X$ be a countable dense subset of $Y$ and $A=\Cnt{b}{X}$. Then $\Sp{}{A}=Y$ and the functional $L(f)=\int_Yf~d\mu$ is 
positive and continuous. But the restriction of $\mu$ on $X$ is the $0$ measure. Therefore $L$ does not have an integral
representation with respect to a measure on $X$.
\end{exm}
\begin{rem}~
\begin{enumerate}
	\item{
	For any positive Radon measure $\mu$ on $\Sp{\norm{}{}}{A}$, the corresponding functional $L_{\mu}$ defined as 
	$L_{\mu}(a)=\int\hat{a}~d\mu$ is $\norm{}{}$-continuous. Since $\Sp{\norm{}{}}{A}$ is compact, 
	$L_{\mu}(1)=\mu(\Sp{\norm{}{}}{A})\leq\infty$, and $L_{\mu}(a)\leq\norm{}{a}L_{\mu}(1)$ which is equivalent to 
	$\norm{}{}$-continuity 	of $L_{\mu}$.
	}
	\item{
	Let $X$ be a compact Hausdorff space, $A\subseteq\Cnt{}{X}$ a unital separating subalgebra and $\mu$ and $\nu$ 
	radon measures on $X$ such that $\int a~d\mu=\int a~d\nu$ for all $a\in A$. Then one can prove that $\mu=\nu$,
	almost every where.
	
	To see this, we first prove that $\supp(\mu)=\supp(\nu)$. Otherwise, there exists a compact set 
	$Z\subseteq X\setminus\supp\mu$ with $\nu(Z)>0$. We can choose $\epsilon$ such that 
	$0<\epsilon<\frac{\nu(Z)}{\mu(X)+\nu(Z)}$. Since $\supp(\mu)$ and $Z$ are compact, there exists a continuous 
	function on $X$ such that $f|_Z=1$ and $f|_{\supp(\mu)}=0$. By Stone-Weierstrass theorem, $\exists a\in A$ such that 
	$|f-a|\leq\epsilon$ on $Z$ and 	$|a|\leq\epsilon$ on $\supp(\mu)$. Replacing $a$ by $a^2$ if necessary, we can assume 
	that $a\ge0$. Therefore $\int a~d\mu\leq\epsilon\mu(X)$ and 
	\[
		(1-\epsilon)\nu(Z)\leq\int_Z a~d\nu\leq\int a~d\nu.
	\] 
	Thus $\nu(Z)\leq(\mu(\supp(\mu))+\nu(Z))\epsilon$, a contradiction.
	Similarly, we can prove that $\int g~d\mu=\int g~d\nu$ for all $g\in\Cnt{}{\supp(\mu)}$ and hence $\mu=\nu$, 
	almost everywhere.
	}
\end{enumerate}
\end{rem}
\textbf{Acknowledgements}: The author wishes to thank Murray Marshall for his useful comments and ideas which formed 
the materials presented in section \ref{lmc2d}.
\bibliographystyle{plain}
\bibliography{IntegralRep}

\begin{thebibliography}{10}

\bibitem{Bayer-Teichmann}
C.~Bayer and J.~Teichmann.
\newblock The proof of {T}chakaloff’s theorem.
\newblock {\em Proc. Amer. Math. Soc.}, 134:3035--3040, 2006.

\bibitem{chq}
G.~Choquet.
\newblock {\em Lectures on Analysis Volume II: Representation Theory}, volume~2
  of {\em Mathematics Lecture Note Series}.
\newblock W.A. Benjamin, Inc, London, 1969.

\bibitem{Curto-Fialkow01}
R.~E. Curto and L.~A. Fialkow.
\newblock An analogue of the {R}iesz-{H}aviland theorem for the truncated
  moment problem.
\newblock {\em J. Funct. Anal.}, 255:2709--2731, 2008.

\bibitem{Daniell}
P.~J. Daniell.
\newblock A general form of integral.
\newblock {\em Ann. of Math.}, 19(4):279--294, 1918.

\bibitem{MGHSKMM}
M.~Ghasemi, S.~Kuhlmann, and M.~Marshall.
\newblock Application of {J}acobi's representation theorem to locally
  multiplicatively convex topological real algebras.
\newblock {\em J. Funct. Anal.}, 266(2):1041--1049, 2014.

\bibitem{Hav01}
E.~K. Haviland.
\newblock On the momentum problem for distribution functions in more than one
  dimension.
\newblock {\em Amer. J. Math.}, 57:562--572, 1935.

\bibitem{Hav02}
E.~K. Haviland.
\newblock On the momentum problem for distribution functions in more than one
  dimension {II}.
\newblock {\em Amer. J. Math.}, 58:164--168, 1936.

\bibitem{Kakutani}
S.~Kakutani.
\newblock Concrete representation of abstract ($m$)-spaces. (a characterization
  of the space of continuous functions.).
\newblock {\em Ann. of Math.}, 42(2):994--1024, 1941.

\bibitem{Markov}
A.~Markov.
\newblock On mean values and exterior densities.
\newblock {\em Rec. Math. Moscou, n. Ser.}, 4:165--190, 1938.

\bibitem{Mar02}
M.~Marshall.
\newblock Approximating positive polynomials using sums of squares.
\newblock {\em Canad. Math. Bull.}, 46:400--418, 2003.

\bibitem{Ped}
G.~K. Pedersen.
\newblock {\em Analysis Now}, volume 118 of {\em Graduate Texts in
  Mathematics}.
\newblock Springer--Verlag, New York Inc., 1989.

\bibitem{Riesz}
F.~Riesz.
\newblock Sur les op{\'{e}}rations fonctionnelles lin{\'{e}}aires.
\newblock {\em C. R. Acad. Sci. Paris}, 149:974--977, 1909.

\end{thebibliography}
\end{document}